\documentclass[12pt]{amsart}

\usepackage[margin=1in]{geometry}
\usepackage{amsmath,amssymb,amsthm,mathtools}
\usepackage{microtype}
\usepackage{enumitem}
\numberwithin{equation}{section}
\usepackage[colorlinks=true,linkcolor=blue,citecolor=blue,urlcolor=blue]{hyperref}

\newtheorem{theorem}{Theorem}[section]
\newtheorem{proposition}[theorem]{Proposition}
\newtheorem{lemma}[theorem]{Lemma}
\newtheorem{corollary}[theorem]{Corollary}
\theoremstyle{definition}
\newtheorem{definition}[theorem]{Definition}
\newtheorem{remark}[theorem]{Remark}

\newcommand{\Sym}{\operatorname{Sym}}
\newcommand{\tr}{\operatorname{tr}}
\newcommand{\Id}{\mathrm{Id}}
\newcommand{\R}{\mathbb{R}}
\newcommand{\Refl}{\mathcal{R}}
\newcommand{\Gplus}{\Gamma_{+}}
\newcommand{\ka}{\underline{k}}

\newcommand{\inner}[2]{\left\langle #1,#2\right\rangle}
\newcommand{\norm}[1]{\left\lvert #1\right\rvert}

\title[Finite-Shift Inverse-Concavity and Exterior Noncollapsing]{Finite-Shift Inverse-Concavity and Exterior Noncollapsing for Fully Nonlinear Curvature Flows of Nonconvex Hypersurfaces}

\author{Weimin Sheng}
\address{Weimin Sheng: School of Mathematical Sciences, Zhejiang University, Hangzhou 310058, China.}
\email{weimins@zju.edu.cn}
\author{Ye Zhu}
\address{Ye Zhu: School of Mathematical Sciences, Zhejiang University, Hangzhou 310058, China.}
\email{ye.zhu@zju.edu.cn}

\subjclass[2020]{53C21 (primary); 35K55, 58J35 (secondary).}
\keywords{Finite-shift inverse-concavity, exterior noncollapsing, fully nonlinear curvature flows, two-point maximum principle.}
\date{}

\begin{document}
\maketitle

\begin{abstract}
We give a finite negative-shift criterion for preserving a prescribed
exterior noncollapsing bound under fully nonlinear curvature flows whose
hypersurfaces need not be convex.  The criterion uses the optimized
two-point expression of Andrews and Langford.  We verify it by a uniform
global two-matrix estimate for the explicit speed
\[
 F_{\varepsilon}(A)=\tr A-
 \varepsilon\frac{\norm{A^{\circ}}^{2}}{\tr A}
\]
on the open cone
\[
 \Gamma=\{A\in\Sym(n):\tr A>0, \norm{A^{\circ}}<\tr A\}.
\]
Here $A^\circ$ denotes the trace-free part of $A$.
This cone contains the positive cone and also matrices with a negative
eigenvalue, while the displayed speed is concave and non-convex.
For an explicit range of parameters depending on the initial exterior
bound, the cone and that bound are preserved throughout the smooth
embedded existence interval.  Together with the known interior estimate
for concave speeds, this gives two-sided tangent-ball bounds for this
nonlinear family, including closed nonconvex initial hypersurfaces.
\end{abstract}

\section{Introduction}\label{sec:introduction}

Let $M^n$ be a closed connected manifold, $n\geq2$, and let
$X:M\times[0,T)\to\R^{n+1}$ be a smooth family of embeddings with
outward unit normal $\nu$.  Write $g$, $h$ and $A=g^{-1}h$ for the
induced metric, second fundamental form and Weingarten map, with the
sign convention
\[
 h(u,v)=\inner{\overline\nabla_u\nu}{dX(v)}=g(Au,v).
\]
Thus round spheres have positive principal curvatures.  We denote the
mean curvature by $H$ and the principal curvatures by
$\kappa_1\leq\cdots\leq\kappa_n$.

Let $\Sym(n)$ be the space of real symmetric $n\times n$ matrices,
equipped with the Frobenius inner product and norm.  Let
$F:\Gamma\to(0,\infty)$ be a smooth orthogonally invariant function
on an open orthogonally invariant cone $\Gamma\subset\Sym(n)$.
We evaluate $F$ on the component matrix of $h$ in a $g$-orthonormal
frame, which is also the matrix of $A$ in that frame.  The identity
$F(OAO^{\mathsf T})=F(A)$ for $O\in O(n)$ makes the resulting scalar
speed independent of the frame.  We write it as $F(A)$, or simply $F$,
and consider the curvature flow
\begin{equation}\label{eq:flow}
 \partial_tX=-F(A)\nu.
\end{equation}
Equivalently, $F(A)=f(\kappa_1,\ldots,\kappa_n)$ for the symmetric
function $f(\kappa)=F(\operatorname{diag}\kappa)$.
Throughout, $F$ is homogeneous of degree one and strictly elliptic:
$F(cA)=cF(A)$ for $c>0$, and its derivative matrix $\dot F^{ij}(A)$
is positive definite.  Here
\[
 \dot F^{ij}(A)B_{ij}
 =\left.\frac{d}{ds}\right|_{s=0}F(A+sB),
\]
with repeated indices summed.  A solution is $\Gamma$-admissible if
$A(x,t)\in\Gamma$ everywhere.  The positive definite cone is denoted
by $\Gplus$.

Sheng and Wang~\cite{SW09} introduced a noncollapsing condition for
mean-convex mean curvature flow and proved that the evolving hypersurfaces
admit interior tangent balls of radius bounded below by a fixed positive
multiple of $H^{-1}$.  Andrews~\cite{A12} subsequently gave a direct
two-point maximum principle proof.  His argument preserves both interior
and exterior noncollapsing bounds.

Andrews, Langford and McCoy~\cite{ALM13} extended this approach to fully
nonlinear curvature flows, obtaining interior noncollapsing for concave
speeds and exterior noncollapsing for convex speeds.  Corresponding
estimates in space forms were developed by Andrews, Han, Li and
Wei~\cite{AHLW15}.  Andrews and Langford~\cite{AL16} then weakened
convexity of the speed to inverse-concavity for the exterior estimate,
at the cost of requiring the hypersurfaces to be strictly convex.
For concave, inverse-concave speeds, the two estimates therefore give
two-sided noncollapsing on the positive curvature cone.

Brendle~\cite{B15} obtained sharp noncollapsing estimates near singularities
of mean curvature flow.  Working with the two-sided tangent-ball condition,
Haslhofer and Kleiner~\cite[Theorems~1.8--1.11]{HK17} established local
curvature and convexity estimates, a global convergence theorem, and a
structure theorem for ancient mean curvature flows.  Such estimates also
enter the construction of mean curvature flow with surgery for mean-convex
surfaces in $\R^3$ by Brendle and Huisken~\cite{BH16}.  The underlying
two-point maximum principle has applications beyond evolution equations,
including Brendle's proof of the Lawson conjecture~\cite{B13}.

The discussion preceding Theorem~1.3 of Langford and Lynch~\cite{LL20}
highlights a gap in the results available there.  They observed that those
two-sided noncollapsing results for nonconvex hypersurfaces did not extend
beyond mean curvature flow.  Concavity alone cannot be expected to supply
the missing exterior bound, since some concave speeds fail even to
preserve convexity.  This distinction motivates a condition that
gives exterior control while allowing both a nonconvex speed and negative
principal curvatures.

We introduce such a condition, called finite-shift inverse-concavity.
Under an initial exterior bound and a lower principal-curvature bound
on the admissible cone, it preserves the prescribed exterior estimate.
The main contribution is a uniform global verification of this condition
for the explicit concave, non-convex speed $F_\varepsilon$ in
\eqref{eq:F-explicit} on a cone allowing negative principal curvatures.
We also prove that this cone is preserved and construct closed nonconvex
initial hypersurfaces satisfying the hypotheses.  Given such initial
data with an exterior bound at scale $H$, the parameter $\varepsilon$
can be chosen sufficiently small, depending on that bound.
Combining the exterior estimate with the interior estimate of~\cite{ALM13}
then gives the two-sided tangent-ball bound in
Corollary~\ref{cor:two-sided}.  This conclusion concerns the smooth
embedded flow; no regularity or surgery theory for $F_\varepsilon$ is
asserted here.

We now state the main results.  For $x\ne y$, set
\begin{equation}\label{eq:kxy}
 k(x,y,t)=
 \frac{2\inner{X(x,t)-X(y,t)}{\nu(x,t)}}
 {\norm{X(x,t)-X(y,t)}^2},
 \qquad
 \ka(x,t)=\inf_{y\ne x}k(x,y,t).
\end{equation}
The quantity $\ka$ is the exterior ball curvature.  We call a flow
exterior noncollapsing at scale $F$ if $\ka\geq-aF$ for some fixed
$a>0$; see \cite{AL16,ALM13} for the tangent-ball interpretation.

The structural condition is a comparison between two curvature matrices.
Matrix products are
ordinary products, $\Id$ is the identity matrix, and $P>0$ means that $P$
is positive definite.

\begin{definition}[finite-shift inverse-concavity]\label{def:shift}
Let $b\geq0$.  A positive homogeneous elliptic speed $F$ satisfies
$\mathcal S_b$ on $\Gamma$ if, for every $A,B\in\Gamma$ and every
$z\in\R$ such that
\[
z\geq-bF(A),\qquad P:=A-z\Id>0,\qquad Q:=B-z\Id>0,
\]
one has
\begin{equation}\label{eq:Ez-def}
 \mathcal E_z(A,B):=
 F(B)-F(A)+\dot F^{ij}(A)(PQ^{-1}P-P)_{ij}\geq0.
\end{equation}
\end{definition}

The expression in \eqref{eq:Ez-def} comes from the optimized two-point
argument of Andrews and Langford~\cite[equation~(2.6) and
Proposition~2.3]{AL16}.  The condition $\mathcal S_b$ requires this
comparison to hold also at negative shifts of depth controlled by $F(A)$.
Our use of the expression combines these shifts with a cone bound that
excludes diagonal contact.  We compare this condition with convexity and
ordinary inverse-concavity after establishing the required matrix
identities in Section~\ref{sec:preliminaries}.

The constant $a$ in the next theorem is the exterior bound to be
preserved, while $b>a$ provides room in the structural shift range.

\begin{theorem}[finite-shift exterior noncollapsing]\label{thm:general}
Let $0<a<b$.  Suppose that $F$ satisfies $\mathcal S_b$ on $\Gamma$ and
that
\begin{equation}\label{eq:cone-gap-general}
 \lambda_{\min}(A)+aF(A)\geq0
 \qquad\text{for every }A\in\Gamma.
\end{equation}
If a compact embedded $\Gamma$-admissible solution of \eqref{eq:flow}
satisfies
\begin{equation}\label{eq:initial-nc}
 \ka(x,0)\geq-aF(x,0)\qquad\text{for every }x\in M,
\end{equation}
then
\begin{equation}\label{eq:preserved-nc}
 \ka(x,t)\geq-aF(x,t)
 \qquad\text{for every }(x,t)\in M\times[0,T).
\end{equation}
\end{theorem}

Here is a completely explicit non-convex speed to which the theorem applies.
For $A\in\Sym(n)$ write
\[
A^{\circ}=A-\frac{\tr A}{n}\Id.
\]

\begin{theorem}[explicit non-convex speed on a cone allowing negative curvatures]
\label{thm:explicit}
Let $n\geq2$, set
\[
 \beta_n=\sqrt{\frac{n-1}{n}}-\frac1n,
\]
and choose $a>\beta_n$.  Put
\[
 C_{n,a}=4n\bigl(1+2n(a+1)\bigr)+\frac{n-1}{n},
\]
and assume
\begin{equation}\label{eq:epsilon-range}
 0<\varepsilon<
 \min\left\{1-\frac{\beta_n}{a},\frac1{C_{n,a}}\right\}.
\end{equation}
Define
\begin{equation}\label{eq:Gamma-explicit}
 \Gamma=\left\{A\in\Sym(n):\tr A>0,
 \ \norm{A^{\circ}}<\tr A\right\}
\end{equation}
and
\begin{equation}\label{eq:F-explicit}
 F_{\varepsilon}(A)=\tr A-
 \varepsilon\frac{\norm{A^{\circ}}^2}{\tr A}.
\end{equation}
Then $F_{\varepsilon}$ is positive, homogeneous of degree one, and strictly
elliptic on $\Gamma$.  It is concave and is not convex.  It satisfies
$\mathcal S_{a+1}$, and
\[
 \lambda_{\min}(A)+aF_{\varepsilon}(A)>0
 \qquad(A\in\Gamma).
\]
Consequently every smooth compact embedded solution of
\[
 \partial_tX=-F_{\varepsilon}(A)\nu
\]
which is initially $\Gamma$-admissible and satisfies
$\ka\geq-aF_{\varepsilon}$ remains $\Gamma$-admissible and continues to
satisfy that estimate throughout its smooth embedded existence interval.
In particular, this is an exterior noncollapsing theorem for hypersurfaces
which may have negative principal curvature.
\end{theorem}

The dependence of $\varepsilon$ on $a$ is essential in this sufficient
parameter range.  For fixed $0<\varepsilon<1$, condition
\eqref{eq:epsilon-range} requires
\[
 \frac{\beta_n}{1-\varepsilon}<a<
 \frac{\varepsilon^{-1}-8n^2-4n-(n-1)/n}{8n^2}.
\]
Thus one cannot accommodate an arbitrary initial exterior constant by
increasing $a$ while keeping $\varepsilon$ fixed.  In
Remark~\ref{rem:nonconvex-data}, the initial hypersurface and $a$ are
chosen first, followed by $\varepsilon$.

\begin{corollary}[Two-sided noncollapsing]\label{cor:two-sided}
Under the hypotheses of Theorem~\ref{thm:explicit}, define
\[
 \overline{k}(x,t)=\sup_{y\ne x}k(x,y,t),\qquad
 C_{\mathrm{in}}=\sup_M
 \frac{\overline{k}(\cdot,0)}{F_\varepsilon(\cdot,0)}.
\]
Then $0<C_{\mathrm{in}}<\infty$ and
\[
 -aF_\varepsilon\leq\ka\leq\overline{k}
 \leq C_{\mathrm{in}}F_\varepsilon
\]
throughout the smooth embedded existence interval.  In particular, at
each point there are interior and exterior tangent balls of radii at least
\[
 \frac{\alpha}{F_\varepsilon},\qquad
 \alpha=\min\{C_{\mathrm{in}}^{-1},a^{-1}\}>0.
\]
\end{corollary}

The paper is organized as follows.  Section~\ref{sec:preliminaries}
fixes notation and establishes the evolution equations, matrix identities
and diagonal compactification used in the proofs.  Section~\ref{sec:geometry}
proves Theorem~\ref{thm:general} by a two-point maximum principle.
Section~\ref{sec:explicit-speed} verifies the finite-shift condition and
cone preservation for $F_\varepsilon$, proves
Theorem~\ref{thm:explicit} and Corollary~\ref{cor:two-sided}, and
exhibits closed nonconvex initial hypersurfaces covered by the theorem.

\section{Preliminaries}\label{sec:preliminaries}

We begin with the evolution of the speed, then establish the matrix
comparisons used in the two-point argument.  We then extend the
two-point function to diagonal directions so that the first-contact
argument takes place on a fixed compact space.

\subsection{Notation and evolution equations}

We use the metric, second fundamental form, and Weingarten map introduced
in Section~\ref{sec:introduction}.  Let $\nabla$ denote the Levi-Civita
connection of $g$.  Indices are raised and lowered using $g$, and repeated
indices are summed.  Matrix contractions and norms are Euclidean in a
$g$-orthonormal frame.  For a self-adjoint endomorphism we also write
$A(u,v)=g(Au,v)$.  The Gauss--Weingarten and Codazzi identities are
\begin{align}
 \overline\nabla_u dX(v)&=dX(\nabla_uv)-h(u,v)\nu,
 \qquad \overline\nabla_u\nu=dX(Au),\notag\\
 \nabla_i h_{jk}&=\nabla_j h_{ik}.\label{eq:gauss-weingarten}
\end{align}
We identify tangent vectors with their images under $dX$ when taking
ambient inner products.

The second matrix derivative, supplementing the first derivative defined
in the Introduction, is
\begin{equation}\label{eq:matrix-derivatives}
 \ddot F^{ij,kl}(A)B_{ij}C_{kl}
 =\left.\frac{\partial^2}{\partial s\,\partial r}\right|_{s=r=0}
 F(A+sB+rC).
\end{equation}
Along the flow, evaluation at $A(x,t)$ is understood unless indicated
otherwise.  We set
\begin{equation}\label{eq:weighted-notation}
 \Delta_F=\dot F^{ij}\nabla_i\nabla_j,
 \qquad |A|_F^2=\dot F^{ij}h_i{}^k h_{kj}.
\end{equation}
If $\dot f_i=\partial f/\partial\kappa_i$, then in a principal frame
$\dot F^{ij}=\dot f_i\delta_{ij}$.  Homogeneity and the chain rule give
\begin{align}
 \dot F^{ij}h_{ij}&=F,\notag\\
 \nabla_iF&=\dot F^{kl}\nabla_i h_{kl},\notag\\
 \nabla_i\nabla_jF
 &=\dot F^{kl}\nabla_i\nabla_j h_{kl}
 +\ddot F^{kl,pq}\nabla_i h_{kl}\nabla_j h_{pq}.
 \label{eq:curvature-chain-rule}
\end{align}

\begin{lemma}[Basic evolution equations]\label{lem:basic-evolution}
Under \eqref{eq:flow}, in fixed coordinates on $M$,
\begin{align}
 \partial_tg_{ij}&=-2Fh_{ij},\qquad
 \partial_t\nu=\nabla F,\notag\\
 \partial_th_{ij}&=\nabla_i\nabla_jF-Fh_i{}^k h_{kj},\notag\\
 \partial_tA&=(\nabla^2F)^\sharp+FA^2,
 \label{eq:basic-evolution}
\end{align}
where the sharp raises one index using $g$.  Moreover,
\begin{equation}\label{eq:F-evolution}
 (\partial_t-\Delta_F)F=F|A|_F^2.
\end{equation}
\end{lemma}

\begin{proof}
Write $X_i=\partial_iX$.  Differentiating $g_{ij}=\langle X_i,X_j\rangle$ and
$\langle\nu,X_i\rangle=0$ gives the metric and normal equations.
Differentiating $h_{ij}=\langle\partial_i\nu,X_j\rangle$ then gives
$\partial_th_{ij}=\nabla_i\nabla_jF-Fh_i{}^kh_{kj}$.
Since $\partial_tg^{ij}=2Fh^{ij}$, raising an index yields the
asserted evolution of $A$.

Choose a moving orthonormal frame with $\partial_te_i=FAe_i$.
In this frame,
\[
 \frac{d}{dt}h(e_i,e_j)
 =(\nabla^2F)(e_i,e_j)+Fg(Ae_i,Ae_j).
\]
The matrix chain rule gives \eqref{eq:F-evolution}.
\end{proof}

\subsection{Matrix identities}

For a twice differentiable curve $A(s)$ in $\Gamma$, the chain rule gives
\begin{align}
 \frac{d}{ds}F(A(s))
 &=\dot F^{ij}(A(s))A'_{ij}(s),\label{eq:matrix-chain-first}\\
 \frac{d^2}{ds^2}F(A(s))
 &=\ddot F^{ij,kl}(A(s))A'_{ij}(s)A'_{kl}(s)
   +\dot F^{ij}(A(s))A''_{ij}(s).\label{eq:matrix-chain-second}
\end{align}
The same derivative convention will be used for other smooth matrix
functions.  In particular, if $P(s)$ is an invertible matrix curve,
differentiating $P(s)P(s)^{-1}=\Id$ gives
\begin{equation}\label{eq:inverse-matrix-derivative}
 \frac{d}{ds}P(s)^{-1}=-P(s)^{-1}P'(s)P(s)^{-1}.
\end{equation}

\begin{lemma}[Shifted comparison identities]\label{lem:shift-identities}
Let $A,B\in\Gamma$ and $z\in\R$ satisfy
$P=A-z\Id>0$ and $Q=B-z\Id>0$.  Put $W=B-A$.  Then
\begin{equation}\label{eq:schur-identity}
 PQ^{-1}P-P=-W+WQ^{-1}W,
\end{equation}
and hence
\begin{equation}\label{eq:Ez-bregman}
 \mathcal E_z(A,B)
 =F(B)-F(A)-\dot F^{ij}(A)W_{ij}
 +\dot F^{ij}(A)(WQ^{-1}W)_{ij}.
\end{equation}
For fixed $A,z$ and $C\in\Sym(n)$, the function
$s\mapsto\mathcal E_z(A,A+sC)$ has value and first derivative zero
at $s=0$, and
\begin{equation}\label{eq:local-shift-condition}
 \left.\frac{d^2}{ds^2}\right|_{s=0}\mathcal E_z(A,A+sC)
 =\ddot F^{ij,kl}(A)C_{ij}C_{kl}
 +2\dot F^{ij}(A)(CP^{-1}C)_{ij}.
\end{equation}
If $F$ satisfies $\mathcal S_b$ and $z\geq-bF(A)$, this last expression
is nonnegative.  Moreover, for fixed $A,B$,
\begin{equation}\label{eq:shift-monotonicity}
 \frac{\partial}{\partial z}\mathcal E_z(A,B)
 =\dot F^{ij}(A)\bigl[W(B-z\Id)^{-2}W\bigr]_{ij}\geq0
\end{equation}
whenever $P,Q>0$.
\end{lemma}

\begin{proof}
Since $P=Q-W$, ordinary matrix multiplication gives
\[
 PQ^{-1}P-P
 =(Q-W)Q^{-1}(Q-W)-(Q-W)
 =-W+WQ^{-1}W.
\]
Substitution into the definition of $\mathcal E_z$ proves
\eqref{eq:Ez-bregman}.  Moreover, $WQ^{-1}W\geq0$, since for every $v$,
\[
 v^{\mathsf T}WQ^{-1}Wv=(Wv)^{\mathsf T}Q^{-1}(Wv)\geq0.
\]

For the local identity, openness implies $A+sC\in\Gamma$ and
$P+sC>0$ for all sufficiently small $|s|$.  Substituting $W=sC$ into
\eqref{eq:Ez-bregman} gives the exact formula
\begin{align*}
 \mathcal E_z(A,A+sC)
 ={}&F(A+sC)-F(A)-s\dot F^{ij}(A)C_{ij}\\
 &+s^2\dot F^{ij}(A)\bigl[C(P+sC)^{-1}C\bigr]_{ij}.
\end{align*}
The constant and linear terms vanish, and differentiating twice at
$s=0$ proves \eqref{eq:local-shift-condition}.
If $F$ satisfies $\mathcal S_b$ and $z\geq-bF(A)$, the shift bound
is unchanged along this variation.  Hence
$\mathcal E_z(A,A+sC)\geq0$, with equality at $s=0$, so its second
derivative there is nonnegative.

Finally, differentiation of \eqref{eq:Ez-bregman} in $z$ gives
\eqref{eq:shift-monotonicity}, since
\[
 WQ^{-2}W=(Q^{-1}W)^{\mathsf T}(Q^{-1}W)\geq0.
\]
\end{proof}

Thus deeper negative shifts impose a stronger condition:
$\mathcal S_{b_2}$ implies $\mathcal S_{b_1}$ if
$0\leq b_1\leq b_2$.  For a fixed pair $A,B$ admitting any shift in
Definition~\ref{def:shift}, the lowest shift $z=-bF(A)$ is itself
admissible.  By \eqref{eq:shift-monotonicity}, it suffices to check
$\mathcal E_{-bF(A)}(A,B)\geq0$ for all pairs with
$A+bF(A)\Id>0$ and $B+bF(A)\Id>0$.  If either positivity condition
fails, that pair admits no permitted shift.

\begin{lemma}[The inverse-dual supporting inequality]
\label{lem:inverse-support}
Suppose that $F$ is defined on the positive definite cone, and set
\[
 F_*(R)=F(R^{-1})^{-1},\qquad R>0.
\]
Then $F_*$ is concave if and only if, for all $A,B>0$,
\begin{equation}\label{eq:inverse-concavity-support}
 \dot F^{ij}(A)(AB^{-1}A)_{ij}\geq\frac{F(A)^2}{F(B)}.
\end{equation}
\end{lemma}

\begin{proof}
The inverse derivative and scalar reciprocal rule give, for
$R=A^{-1}>0$ and every $S\in\Sym(n)$,
\begin{equation}\label{eq:inverse-dual-derivative}
 \dot F_*^{ij}(R)S_{ij}
 =F(A)^{-2}\dot F^{ij}(A)(ASA)_{ij}.
\end{equation}
Since $F_*$ is homogeneous of degree one, Euler's identity gives
$\dot F_*^{ij}(R)R_{ij}=F_*(R)$.  Consequently
\eqref{eq:inverse-concavity-support}, with $S=B^{-1}$, is equivalent to
\[
 F_*(S)\leq F_*(R)+\dot F_*^{ij}(R)(S-R)_{ij}
 \qquad(R,S>0).
\]
On the convex positive definite cone, these global supporting-plane
inequalities are equivalent to concavity.
\end{proof}

The following block-matrix fact will justify the inverse of the
shifted second-point curvature matrix at a contact pair.

\begin{lemma}[Block-matrix minimization]\label{lem:block-minimum}
Let $P,Q,C\in\Sym(n)$ with $P>0$.  If
\[
 \begin{pmatrix}C&-P\\-P&Q\end{pmatrix}\geq0,
\]
then $Q>0$ and $C-PQ^{-1}P\geq0$.  For any $Q>0$ and any
$C\in\Sym(n)$, the exact identity
\begin{align}
 \langle Cu,u\rangle-2\langle Pu,v\rangle+\langle Qv,v\rangle
 ={}&\langle(C-PQ^{-1}P)u,u\rangle\notag\\
 &+\langle Q(v-Q^{-1}Pu),v-Q^{-1}Pu\rangle
 \label{eq:block-minimum}
\end{align}
holds for all $u,v\in\R^n$.  In particular, its minimum over $v$ is attained at
$v=Q^{-1}Pu$.
\end{lemma}

\begin{proof}
The block-matrix inequality means that
\[
 \langle Cu,u\rangle-2\langle Pu,v\rangle+\langle Qv,v\rangle\geq0
 \qquad\text{for all }u,v\in\R^n.
\]
Taking $u=0$ gives $\langle Qv,v\rangle\geq0$ for every $v$, so
$Q\geq0$.  To prove strict positivity, let $v\in\ker Q$.
Testing the same inequality on $(u,tv)$ gives
\[
 \langle Cu,u\rangle-2t\langle Pu,v\rangle\geq0
 \qquad\text{for every }u\in\R^n\text{ and }t\in\R.
\]
For fixed $u$, a nonzero coefficient $\langle Pu,v\rangle$ would make
the left-hand side negative for $t$ of the appropriate sign and
sufficiently large magnitude.  Hence $\langle Pu,v\rangle=0$ for
every $u$.  Since $P$ is symmetric, this is
$\langle u,Pv\rangle=0$ for every $u$, and therefore $Pv=0$.
The positive definiteness of $P$ forces $v=0$.  Thus $Q$ is symmetric
positive semidefinite with trivial kernel, which proves $Q>0$.

We next prove the identity for arbitrary $Q>0$, independently of the
block-matrix inequality.  Symmetry of $P$, $Q$ and $Q^{-1}$ gives
\begin{align*}
 &\langle Q(v-Q^{-1}Pu),v-Q^{-1}Pu\rangle\\
 &\quad=\langle Qv,v\rangle-\langle Qv,Q^{-1}Pu\rangle
        -\langle Pu,v\rangle+\langle Pu,Q^{-1}Pu\rangle\\
 &\quad=\langle Qv,v\rangle-2\langle Pu,v\rangle
        +\langle PQ^{-1}Pu,u\rangle.
\end{align*}
Here $\langle Qv,Q^{-1}Pu\rangle=\langle v,Pu\rangle$ and
$\langle Pu,Q^{-1}Pu\rangle=\langle u,PQ^{-1}Pu\rangle$;
no commutation of $P$ and $Q$ is used.  Adding
$\langle(C-PQ^{-1}P)u,u\rangle$ proves
\eqref{eq:block-minimum}.

For each fixed $u$, the last term in \eqref{eq:block-minimum} is
nonnegative and, since $Q>0$, vanishes exactly when $v=Q^{-1}Pu$.
Consequently,
\[
 \min_{v\in\R^n}
 \bigl\{\langle Cu,u\rangle-2\langle Pu,v\rangle
                 +\langle Qv,v\rangle\bigr\}
 =\langle(C-PQ^{-1}P)u,u\rangle,
\]
and the minimizer is unique.  Under the block-matrix inequality, the
expression being minimized is nonnegative.  Its minimum is therefore
nonnegative for every $u$, proving $C-PQ^{-1}P\geq0$.
\end{proof}

We can now compare the finite-shift condition with the familiar
convexity assumptions.

\begin{proposition}[Comparison of the structural conditions]\label{prop:hierarchy}
Assume in this proposition that $\Gamma$ is convex and contains the positive
definite cone.  All other assumptions on $F$ are as in Section~\ref{sec:introduction}.
\begin{enumerate}[label=\textup{(\roman*)}]
\item If $F$ is convex, then it satisfies $\mathcal S_b$ for every $b\geq0$.
\item The condition $\mathcal S_0$, restricted to positive definite
matrices, is equivalent to ordinary inverse-concavity, namely concavity of
\[
 F_*(R)=\frac1{F(R^{-1})},\qquad R>0.
\]
\item Ordinary inverse-concavity does not imply $\mathcal S_b$ for any
fixed $b>0$.
\item If the same speed satisfies $\mathcal S_b$ for every $b>0$, then it is
convex.
\end{enumerate}
\end{proposition}

\begin{proof}
For (i), let $P=A-z\Id>0$, $Q=B-z\Id>0$ and $W=B-A$.
Convexity gives $F(B)-F(A)-\dot F^{ij}(A)W_{ij}\geq0$.
The remaining term in \eqref{eq:Ez-bregman} is nonnegative because
$\dot F^{ij}(A)>0$ and $WQ^{-1}W\geq0$.  Thus $\mathcal E_z\geq0$.

For (ii), suppose first that $F_*$ is concave.  Euler's identity and
Lemma~\ref{lem:inverse-support} give
\begin{align*}
 \mathcal E_0(A,B)
 &=F(B)-2F(A)+\dot F^{ij}(A)(AB^{-1}A)_{ij}\\
 &\geq F(B)-2F(A)+\frac{F(A)^2}{F(B)}
 =\frac{(F(B)-F(A))^2}{F(B)}\geq0.
\end{align*}
For every allowed $z\geq0$, \eqref{eq:shift-monotonicity} gives
$\mathcal E_z(A,B)\geq\mathcal E_0(A,B)\geq0$.
Conversely, apply $\mathcal S_0$ at zero shift to $A$ and
$\frac{F(A)}{F(B)}B$.  Homogeneity gives
\[
 0\leq\mathcal E_0\left(A,\frac{F(A)}{F(B)}B\right)
 =-F(A)+\frac{F(B)}{F(A)}\dot F^{ij}(A)(AB^{-1}A)_{ij}.
\]
This is \eqref{eq:inverse-concavity-support}, which implies concavity of
$F_*$ by Lemma~\ref{lem:inverse-support}.

For (iii), in any dimension $n\geq2$, take the harmonic-mean speed
\[
 F(A)=\bigl(\tr A^{-1}\bigr)^{-1},\qquad A>0.
\]
It is positive and homogeneous of degree one, and
$\dot F^{ij}(A)=F(A)^2(A^{-2})_{ij}$ is positive definite.
Its inverse dual is $F_*(R)=\tr R$, so it is inverse-concave.
At $A=\Id$, choose $C=\operatorname{diag}(1,-1,0,\ldots,0)$.
Then
\[
 F(A)=\frac1n,\qquad
 \dot F^{ij}(A)=\frac1{n^2}\delta_{ij},
\]
while
\[
 F(A+sC)
 =\left(n-2+\frac1{1+s}+\frac1{1-s}\right)^{-1}
 =\frac1n-\frac2{n^2}s^2+O(s^4).
\]
Thus $\ddot F^{ij,kl}(A)C_{ij}C_{kl}=-4/n^2$.
For any $b>0$, choose $z=-bF(A)=-b/n$, so
$P=(1+b/n)\Id$.  By Lemma~\ref{lem:shift-identities}, $\mathcal S_b$ would require
the expression in \eqref{eq:local-shift-condition} to be nonnegative,
but here it equals
\[
 -\frac4{n^2}+\frac4{n^2(1+b/n)}<0.
\]
This contradicts the necessary local condition and proves (iii).

For (iv), fix arbitrary $A,B\in\Gamma$ and put $W=B-A$.  For every
sufficiently large $s>0$, both $A+s\Id$ and $B+s\Id$ are positive
definite.  We may therefore apply $\mathcal S_{s/F(A)}$ with $z=-s$.
Using \eqref{eq:Ez-bregman}, we obtain
\[
 F(B)-F(A)-\dot F^{ij}(A)W_{ij}
 +\dot F^{ij}(A)\bigl[W(B+s\Id)^{-1}W\bigr]_{ij}\geq0.
\]
The last term tends to zero as $s\to\infty$, and hence
\[
 F(B)\geq F(A)+\dot F^{ij}(A)(B-A)_{ij}
 \qquad(A,B\in\Gamma).
\]
On the convex domain $\Gamma$, these global supporting-plane inequalities
imply that $F$ is convex.  This proves (iv).

\end{proof}

\subsection{The diagonal limit of the two-point function}

We make precise the limiting behaviour of \eqref{eq:kxy} at the diagonal,
as in \cite[Section~12.1]{ACGL}.  Fix a reference metric $g_0$ on $M$.
Replace each point $(x,x)$ of the diagonal in $M\times M$ by the sphere
of $g_0$-unit tangent directions at $x$.  Near this new boundary use
coordinates
\[
 (x,v,s)\longmapsto(x,\exp_x^{g_0}(sv)),
 \qquad |v|_{g_0}=1,\quad 0<s<\delta,
\]
where $\delta$ is smaller than the injectivity radius of $g_0$.
Adjoining $s=0$ gives the real blow-up $\widehat{M\times M}$, a compact
manifold with boundary $S_{g_0}M$, fixed in time.

For $y=\exp_x^{g_0}(sv)$, Taylor expansion gives
\begin{align*}
 2\inner{X(x,t)-X(y,t)}{\nu(x,t)}
 &=s^2h_t(v,v)+O(s^3),\\
 \norm{X(x,t)-X(y,t)}^2
 &=s^2g_t(v,v)+O(s^3).
\end{align*}
Indeed, the first normal derivative vanishes, and
\eqref{eq:gauss-weingarten} gives the normal acceleration
$-h_t(v,v)\nu(x,t)$.  The tangential acceleration caused by using $g_0$
instead of $g_t$ does not contribute to the numerator.  Hence the
continuous boundary value is
\begin{equation}\label{eq:diagonal-limit}
 k(x,v,t)=\frac{h_t(v,v)}{g_t(v,v)}
 \qquad ((x,v)\in S_{g_0}M).
\end{equation}
The remainders are uniform on compact time intervals, where $g_t(v,v)$
is bounded away from zero.  Thus the extension is jointly continuous in
space and time.  At fixed $x$, its boundary values range from
$\kappa_1(x,t)$ to $\kappa_n(x,t)$, so $\ka(x,t)\leq\kappa_1(x,t)$
and the infimum defining $\ka$ is a minimum on the compactified fibre.
Consequently every global extremum used in the first-contact argument
is attained on a fixed compact space, either off the diagonal or at a
boundary direction.

\section{Exterior noncollapsing}\label{sec:geometry}

We prove Theorem~\ref{thm:general}, following the optimized two-point
argument of Andrews and Langford~\cite{AL16} and retaining the details
needed when the principal curvatures need not be positive.
The diagonal compactification was
constructed in Section~\ref{sec:preliminaries}; it remains to control
off-diagonal contact points.

\begin{lemma}[two-point evolution inequality]\label{lem:two-point}
Let $0<t_0<T$, and let $\phi$ be a smooth lower parabolic support for
$\ka$ at $(x_0,t_0)$:
in a backward space--time neighbourhood, assume
\(\phi(x,t)\leq\ka(x,t)\), with equality at \((x_0,t_0)\).
Suppose that this equality is realized by $y_0\ne x_0$ and that
\[
 z:=\phi(x_0,t_0)=k(x_0,y_0,t_0)
 <\min\{0,\lambda_{\min}(A)\},
 \qquad A=A(x_0,t_0).
\]
Let $B$ be the Weingarten map at $y_0$, transported to $T_{x_0}M$ by the
reflection described in the proof.  Then $Q=B-z\Id$ is positive definite,
and
\begin{equation}\label{eq:two-point-main}
 (\partial_t-\Delta_F)\phi\geq
 z\,|A|_F^2+
 \frac{2}{d^2}\mathcal E_z(A,B)+
 \frac{8}{d^4}\sum_{i=1}^n \dot f_i(\kappa_i-z)\xi_i^2,
\end{equation}
where
\[
 d=\norm{X(x_0)-X(y_0)},\quad
 \xi_i=\inner{X(x_0)-X(y_0)}{e_i},
\]
and $\{e_i\}$ is an eigenbasis of $A$ with eigenvalues $\kappa_i$ and
$\dot f_i=\partial f/\partial\kappa_i$.
\end{lemma}

\begin{proof}
Suppress the time $t_0$, write $\nu_x=\nu(x_0,t_0)$ and
$\nu_y=\nu(y_0,t_0)$, and set $\xi=X(x_0)-X(y_0)$, $d=\norm\xi$.
Spatial derivatives of $F$ below are evaluated at $x_0$.
Since $y_0$ minimizes $k(x_0,\cdot,t_0)$, differentiation in the $y$
variable shows that $\nu_x-z\xi$ is normal at $y_0$.  Moreover
\[
 \norm{\nu_x-z\xi}^2
 =1-2z\inner{\nu_x}{\xi}+z^2d^2=1.
\]
The sign of this normal is fixed by the global support and embeddedness.
Indeed, put \(r=-1/z>0\) and \(c=X(x_0)+r\nu_x\).  Since
\(k(x_0,y,t_0)\geq z\) for every \(y\ne x_0\), one has
\[
 \norm{X(y,t_0)-c}^2-r^2
 =\norm{X(x_0,t_0)-X(y,t_0)}^2
 \left(1-\frac{k(x_0,y,t_0)}{z}\right)\geq0.
\]
For all sufficiently small $\delta>0$, the point
$X(x_0,t_0)+\delta\nu_x$ lies both in this open ball and in the exterior
of the region bounded by $X(M,t_0)$.  The ball is connected and disjoint
from the hypersurface, so it lies entirely in that exterior.
At the second contact \(y_0\), the unit vector pointing into that exterior
ball is
\[
 \frac{c-X(y_0)}r=\nu_x-z\xi.
\]
It is therefore the outward normal \(\nu_y\), and hence
\begin{equation}\label{eq:normal-reflection}
 \nu_y=\nu_x-z\xi.
\end{equation}
Equivalently, the Euclidean reflection
\[
 \Refl=\Id-2\xi\otimes\xi/d^2
\]
sends $\nu_x$ to $\nu_y$ and $T_{x_0}M$ isometrically to $T_{y_0}M$.
We use $\Refl$ to identify the two tangent spaces, and write $B$ for
the pullback of the Weingarten map at $y_0$.  Orthogonal invariance
ensures that $B\in\Gamma$ and $F(B)=F(y_0,t_0)$.

We next compute the derivatives needed at this contact pair.  Put
$P=A-z\Id$.  For $u,v\in T_{x_0}M$, direct differentiation gives
\begin{equation}\label{eq:first-k}
 \nabla^x_u k=\frac{2}{d^2}\inner{\xi}{Pu},
 \qquad \nabla^y_{\Refl v}k=0.
\end{equation}
To obtain the second derivative explicitly, take intrinsic geodesics
with initial velocities $u$ and $\Refl v$ at the two points, and let
primes denote derivatives along this paired variation at the contact.
The Gauss--Weingarten identities give
\begin{align*}
 \xi'&=u-\Refl v,
 &\xi''&=-A(u,u)\nu_x+B(v,v)\nu_y,\\
 \nu_x'&=Au,
 &\nu_x''&=(\nabla_uA)u-\norm{Au}^2\nu_x.
\end{align*}
Differentiating $d^2k=2\inner\xi{\nu_x}$ twice yields
\[
 d^2 k''=2\inner{\xi''}{\nu_x}
 +4\inner{\xi'}{\nu_x'}+2\inner\xi{\nu_x''}
 -z(d^2)''-2(d^2)'k',
\]
where $(d^2)'=2\inner\xi{\xi'}$,
$(d^2)''=2\norm{\xi'}^2+2\inner\xi{\xi''}$, and
$k'=2d^{-2}\inner\xi{Pu}$ by \eqref{eq:first-k}.
Substituting these expressions and \eqref{eq:normal-reflection} gives
\begin{align}
 \nabla^2 k[(u,\Refl v),(u,\Refl v)]
={}&\frac{2}{d^2}\left\{
 A(u,u)-2A(u,v)+B(v,v)-z\norm{u-v}^2\right\}\notag\\
&+z^2A(u,u)-z\norm{Au}^2
 +\frac{2}{d^2}\inner{\xi}{(\nabla_uA)u}\notag\\
&-\frac{8}{d^4}\inner{\xi}{u}\inner{\xi}{Pu}.
 \label{eq:second-k}
\end{align}
In \eqref{eq:second-k}, vectors at $y_0$ are written as their pullbacks
to $T_{x_0}M$ under $\Refl$: we have abbreviated $\inner{Au}{\Refl v}$
by $A(u,v)$ and $\norm{u-\Refl v}^2$ by $\norm{u-v}^2$.  Two reflection
corrections are hidden in this abbreviation.  Since
$\Refl=\Id-2\xi\otimes\xi/d^2$ is symmetric, every ambient vector $w$
satisfies
$\inner{w}{\Refl v}=\inner{w}{v}-2d^{-2}\inner{\xi}{w}\inner{\xi}{v}$,
and hence
\[
 \inner{Au}{\Refl v}
 =A(u,v)-\frac{2}{d^2}\inner{\xi}{Au}\inner{\xi}{v},
 \qquad
 \norm{u-\Refl v}^2
 =\norm{u-v}^2+\frac{4}{d^2}\inner{\xi}{u}\inner{\xi}{v}.
\]
After multiplication by $2/d^2$, these corrections contribute
$\frac{8}{d^4}\inner{\xi}{Pu}\inner{\xi}{v}$ to $\nabla^2k$.  They are
exactly cancelled by the last term of \eqref{eq:second-k}: since
$\inner{\xi}{\Refl v}=\inner{\Refl\xi}{v}=-\inner{\xi}{v}$, one has
\[
 -2(d^2)'k'
 =-\frac{8}{d^4}\bigl(\inner{\xi}{u}+\inner{\xi}{v}\bigr)
 \inner{\xi}{Pu},
\]
whose $\inner{\xi}{v}$-part is the negative of the correction above.
Similarly, differentiating $d^2k=2\inner\xi{\nu_x}$ in time, using
\[
 \partial_t\xi=-F(A)\nu_x+F(B)\nu_y,
 \qquad \partial_t\nu_x=\nabla F(A)
\]
yields
\begin{equation}\label{eq:time-k}
 \partial_t k
 =\frac{2}{d^2}\left(F(B)-F(A)
 +\inner{\xi}{\nabla F(A)}\right)+z^2F(A).
\end{equation}

We now use the spatial minimum of $k-\phi$ to justify the inverse
of $Q=B-z\Id$.  The first line of \eqref{eq:second-k} is $2/d^2$ times
\[
 P(u,u)-2P(u,v)+Q(v,v),
\]
and all remaining terms depend only on $u$, as does the Hessian of
$\phi$.  Thus, after multiplication by $d^2/2$, the nonnegative
Hessian of $k-\phi$ has mixed block $-P$ and $yy$ block $Q$.
Since $P>0$, Lemma~\ref{lem:block-minimum} gives $Q>0$.

We may therefore minimize the displayed quadratic form in $v$.
By \eqref{eq:block-minimum}, its minimum is
$P(u,u)-(PQ^{-1}P)(u,u)$, attained at $v=Q^{-1}Pu$.
Choose a principal frame $\{e_i\}$ and set
\[
 v_i=Q^{-1}Pe_i.
\]
Since $k-\phi$ has a spatial minimum and a backwards space-time minimum at
the contact point,
\[
 \phi_t\geq k_t,\qquad
 \sum_i \dot f_i\nabla^2 k[(e_i,\Refl v_i),(e_i,\Refl v_i)]
 \geq \Delta_F\phi.
\]
The first line of \eqref{eq:second-k}, after contraction, is
\[
 \frac{2}{d^2}\dot F^{ij}(A)(P-PQ^{-1}P)_{ij}.
\]
Euler homogeneity, $\dot F^{ij}(A)A_{ij}=F(A)$, gives
\[
 \sum_i \dot f_i\bigl(z^2A(e_i,e_i)-z\norm{Ae_i}^2\bigr)
 =-z|A|_F^2+z^2F(A).
\]
Codazzi gives
\[
 \sum_i \dot f_i\inner{\xi}{(\nabla_{e_i}A)e_i}
 =\inner{\xi}{\nabla F(A)}.
\]
Finally, the last line of \eqref{eq:second-k} contracts to
\[
 -\frac{8}{d^4}\sum_i \dot f_i(\kappa_i-z)\xi_i^2.
\]
Subtracting the contracted second variation from \eqref{eq:time-k}
cancels both the gradient term and $z^2F(A)$, and yields exactly
\eqref{eq:two-point-main}.
\end{proof}

\begin{proof}[Proof of Theorem~\ref{thm:general}]
Fix $\tau<T$ and choose $0<\eta<(b-a)e^{-\tau}$, so that
$a+\eta e^t<b$ on $[0,\tau]$.
Consider on the blown-up product the function
\[
 \Psi(x,y,t)=k(x,y,t)+(a+\eta e^t)F(x,t).
\]
At $t=0$, \eqref{eq:initial-nc} and $F>0$ give $\Psi>0$.  On the boundary
of the blow-up, \eqref{eq:diagonal-limit} and
\eqref{eq:cone-gap-general} give
\[
 \Psi(x,v,t)=\frac{h_t(v,v)}{g_t(v,v)}+(a+\eta e^t)F(x,t)
 \geq\lambda_{\min}(A)+aF(A)+\eta e^tF(A)>0.
\]
Thus, if $\Psi$ first reached zero, it would do so at an off-diagonal pair
$(x_0,y_0,t_0)$ with $t_0\leq\tau$.

At such a point
\[
 \phi(x,t)=-(a+\eta e^t)F(x,t)
\]
is a lower parabolic support for $\ka$, and
\[
 z=\phi(x_0,t_0)=-(a+\eta e^{t_0})F(A)<0,
 \qquad z
 \geq-bF(A).
\]
The boundary estimate also gives
$\lambda_{\min}(P)\geq\eta e^{t_0}F(A)>0$ for $P=A-z\Id$.
In particular, $z<\lambda_{\min}(A)$, so
Lemma~\ref{lem:two-point} applies and gives $Q>0$.
The shift condition $\mathcal S_b$ is therefore applicable to $A,B,z$.
Every term after the first on the
right-hand side of \eqref{eq:two-point-main} is nonnegative.  On the other
hand, Lemma~\ref{lem:basic-evolution} gives
\[
 (\partial_t-\Delta_F)\phi
 =z\,|A|_F^2-\eta e^{t_0}F.
\]
Comparison with \eqref{eq:two-point-main} yields
\[
 -\eta e^{t_0}F\geq0,
\]
contradicting $\eta>0$ and $F>0$.
Consequently $\Psi>0$ on $[0,\tau]$.  Letting $\eta\downarrow0$ gives
$k(x,y,t)+aF(x,t)\geq0$ for all $x\ne y$ and $t\leq\tau$.  Since
$\tau<T$ was arbitrary, \eqref{eq:preserved-nc} follows.
\end{proof}

\section{The explicit non-convex speed}
\label{sec:explicit-speed}

Throughout this section, the parameters and speed are as in
Theorem~\ref{thm:explicit}.  We first compute the nonlinear part of the
speed, then verify its local properties and global finite-shift bound.
After proving preservation of the cone, we apply the exterior theorem
and deduce the two-sided estimate.

The following lemma gives the derivatives and exact first-order
remainder of the quadratic quotient in the speed.

\begin{lemma}[A quadratic quotient]\label{lem:quadratic-quotient}
On the half-space $\tr A>0$, define
\[
 q(A)=\frac{\norm{A^\circ}^2}{\tr A},\qquad
 h_A=\tr A,\qquad U_A=\frac{A^\circ}{h_A}.
\]
Then, for every $C\in\Sym(n)$,
\begin{align}
 \dot q^{ij}(A)C_{ij}
 &=2\inner{U_A}{C^\circ}-\norm{U_A}^2\tr C,
 \label{eq:quotient-first}\\
 \ddot q^{ij,kl}(A)C_{ij}C_{kl}
 &=\frac2{h_A}\norm{C^\circ-(\tr C)U_A}^2.
 \label{eq:quotient-second}
\end{align}
For any $A,B$ with positive trace, set
$h_B=\tr B$ and $U_B=B^\circ/h_B$.  Then
\begin{equation}\label{eq:perspective-bregman}
 q(B)-q(A)-\dot q^{ij}(A)(B-A)_{ij}
 =h_B\norm{U_B-U_A}^2.
\end{equation}
\end{lemma}

\begin{proof}
Write $c=\tr C$.  Along the line $A+sC$, we have
\[
 q(A+sC)
 =\frac{h_A^2\norm{U_A}^2+2s h_A\inner{U_A}{C^\circ}
              +s^2\norm{C^\circ}^2}{h_A+sc}.
\]
One differentiation gives \eqref{eq:quotient-first}; a second gives
\[
 \left.\frac{d^2}{ds^2}\right|_{s=0}q(A+sC)
 =\frac2{h_A}\left(\norm{C^\circ}^2
       -2c\inner{U_A}{C^\circ}+c^2\norm{U_A}^2\right),
\]
which is \eqref{eq:quotient-second}.  Finally, put $W=B-A$.
Since $W^\circ=h_BU_B-h_AU_A$ and $\tr W=h_B-h_A$,
\begin{align*}
 q(B)-q(A)-\dot q^{ij}(A)W_{ij}
 ={}&h_B\norm{U_B}^2-h_A\norm{U_A}^2
      -2\inner{U_A}{h_BU_B-h_AU_A}\\
    &+\norm{U_A}^2(h_B-h_A)
 =h_B\norm{U_B-U_A}^2.
\end{align*}
\end{proof}

With these identities in hand, we verify the local properties of the
speed and the curvature bound needed to exclude diagonal contact.

\begin{lemma}[ellipticity, concavity, and the cone gap]
\label{lem:speed-properties}
Let $F_{\varepsilon}$ and $\Gamma$ be given by
\eqref{eq:Gamma-explicit}--\eqref{eq:F-explicit}.  Under
\eqref{eq:epsilon-range}, $F_{\varepsilon}$ is positive, strictly elliptic,
concave and non-convex on $\Gamma$, and satisfies
\begin{equation}\label{eq:explicit-cone-gap}
 \lambda_{\min}(A)+aF_{\varepsilon}(A)>0
 \qquad(A\in\Gamma).
\end{equation}
\end{lemma}

\begin{proof}
The set $\Gamma$ is open by continuity of trace and norm, and it is
closed under multiplication by positive scalars.  If $A,B\in\Gamma$ and
$0<\theta<1$, then $\tr((1-\theta)A+\theta B)>0$ and
\begin{align*}
 \norm{((1-\theta)A+\theta B)^\circ}
 &=\norm{(1-\theta)A^\circ+\theta B^\circ}\\
 &\leq(1-\theta)\norm{A^\circ}+\theta\norm{B^\circ}\\
 &<(1-\theta)\tr A+\theta\tr B
 =\tr((1-\theta)A+\theta B).
\end{align*}
The endpoint cases are immediate, so $\Gamma$ is convex.
For $A>0$, put $h_A=\tr A=\sum_i\lambda_i>0$.  Since
\[
 h_A^2-\norm A^2=2\sum_{i<j}\lambda_i\lambda_j>0,
 \qquad \norm{A^\circ}^2=\norm A^2-\frac{h_A^2}{n}<h_A^2,
\]
we have $A\in\Gamma$.
For a general $A\in\Gamma$, put $h_A=\tr A$ and $U_A=A^{\circ}/h_A$.
Then $\tr U_A=0$, $\norm{U_A}<1$, and
\begin{equation}\label{eq:F-U}
 F_{\varepsilon}(A)=h_A(1-\varepsilon\norm{U_A}^2)>0.
\end{equation}
Since $F_\varepsilon(A)=\tr A-\varepsilon q(A)$,
Lemma~\ref{lem:quadratic-quotient} gives
\[
 \dot F_\varepsilon^{ij}(A)=\delta_{ij}-\varepsilon\dot q^{ij}(A),
 \qquad \ddot F_\varepsilon^{ij,kl}(A)
 =-\varepsilon\ddot q^{ij,kl}(A).
\]
In an eigenbasis of $A$, if $u_i$ are the eigenvalues of $U_A$, then
\[
 \dot f_i=1-\varepsilon(2u_i-\norm{U_A}^2).
\]
For each fixed $i$, the trace-free condition and Cauchy--Schwarz give
\[
 u_i^2=\left(\sum_{j\ne i}u_j\right)^2
 \leq(n-1)\sum_{j\ne i}u_j^2.
\]
Dividing by $n-1$ and adding $u_i^2$ yields
\[
 \norm{U_A}^2=u_i^2+\sum_{j\ne i}u_j^2
 \geq\frac{n}{n-1}u_i^2.
\]
Completing the square, we obtain
\[
 2u_i-\norm{U_A}^2
 \leq 2u_i-\frac{n}{n-1}u_i^2
 =\frac{n-1}{n}
 -\frac{n}{n-1}\left(u_i-\frac{n-1}{n}\right)^2
 \leq\frac{n-1}{n}.
\]
Hence $\dot f_i\geq1-\varepsilon(n-1)/n>0$, since
\eqref{eq:epsilon-range} implies $\varepsilon<1$.
This proves strict ellipticity.  Equation~\eqref{eq:quotient-second} gives $\ddot F_\varepsilon\leq0$,
so $F_\varepsilon$ is concave.  At a scalar matrix $A=c\Id$ with $c>0$
and for any nonzero trace-free $C$,
we have $U_A=0$ and
$\ddot F_\varepsilon^{ij,kl}(A)C_{ij}C_{kl}
=-2\varepsilon\norm C^2/h_A<0$.  Hence the speed is not convex.

The same trace-free estimate gives
$|u_i|\leq\sqrt{(n-1)/n}\norm{U_A}<\sqrt{(n-1)/n}$ for every $i$.
Consequently
\[
 \frac{\lambda_{\min}(A)}{h_A}>\frac1n-\sqrt{\frac{n-1}{n}}=-\beta_n.
\]
Together with \eqref{eq:F-U}, this gives
\[
 \frac{\lambda_{\min}(A)+aF_{\varepsilon}(A)}{h_A}
 >-\beta_n+a(1-\varepsilon)>0,
\]
where the last inequality is exactly the first restriction in
\eqref{eq:epsilon-range}.
\end{proof}

\begin{lemma}[global finite-shift estimate]
\label{lem:global-shift}
For any $A,B\in\Gamma$ and any
$z\geq-(a+1) F_{\varepsilon}(A)$ such that
\[
 P=A-z\Id>0,\qquad Q=B-z\Id>0,
\]
one has
\begin{equation}\label{eq:global-E-bound}
 \mathcal E_z(A,B)\geq
 \bigl(1-\varepsilon C_{n,a}\bigr)
 \tr\bigl((B-A)Q^{-1}(B-A)\bigr)\geq0.
\end{equation}
\end{lemma}

\begin{proof}
Write
\[
 h_A=\tr A,\quad h_B=\tr B,\quad
 U_A=\frac{A^{\circ}}{h_A},\quad U_B=\frac{B^{\circ}}{h_B},\quad
 W=B-A,\quad R=WQ^{-1}W.
\]
Since $A,B\in\Gamma$, we have $h_A,h_B>0$,
$\tr U_A=\tr U_B=0$, and $\norm{U_A},\norm{U_B}<1$.
Moreover, $R\geq0$, because $W$ is symmetric and
\[
 \inner{Rv}{v}=\inner{Q^{-1}Wv}{Wv}\geq0
 \qquad(v\in\R^n).
\]

We first expand the shifted comparison expression.  Since
$F_\varepsilon=\tr-\varepsilon q$ and $\tr W=h_B-h_A$,
the linear trace terms cancel, and \eqref{eq:perspective-bregman} gives
\begin{align*}
 &F_\varepsilon(B)-F_\varepsilon(A)
       -\dot F_\varepsilon^{ij}(A)W_{ij}\\
 &\quad=-\varepsilon\bigl(q(B)-q(A)-\dot q^{ij}(A)W_{ij}\bigr)
 =-\varepsilon h_B\norm{U_B-U_A}^2.
\end{align*}
Also,
\[
 \dot F_\varepsilon^{ij}(A)R_{ij}
 =\tr R-\varepsilon\dot q^{ij}(A)R_{ij}.
\]
Substituting these two identities into \eqref{eq:Ez-bregman} yields
\begin{equation}\label{eq:E-explicit-decomp}
 \mathcal E_z(A,B)=
 \tr R-\varepsilon\left(h_B\norm{U_B-U_A}^2+\dot q^{ij}(A)R_{ij}\right).
\end{equation}
The estimate in the proof of Lemma~\ref{lem:speed-properties} gives
\[
 \dot q(A)=2U_A-\norm{U_A}^2\Id\leq\frac{n-1}{n}\Id.
\]
Since $R\geq0$, contraction yields the one-sided bound
\begin{equation}\label{eq:DgR-bound}
 \dot q^{ij}(A)R_{ij}\leq\frac{n-1}{n}\tr R.
\end{equation}
No commutation between $\dot q(A)$ and $R$ is required.

We next bound $h_B\norm{U_B-U_A}^2$ by a multiple of $\tr R$.
The shift assumption and $F_\varepsilon(A)\leq h_A$ imply
$-nz\leq n(a+1)F_\varepsilon(A)\leq n(a+1)h_A$.
Since every eigenvalue of $Q$ is positive,
\begin{align}
 \lambda_{\max}(Q)&\leq\tr Q=h_B-nz
 \leq h_B+n(a+1) h_A,\label{eq:lambdaQ}\\
 \tr R&\geq\frac{\norm W^2}{\lambda_{\max}(Q)}.
 \label{eq:R-lower}
\end{align}
To justify \eqref{eq:R-lower}, the spectral theorem gives
$Q^{-1}\geq\lambda_{\max}(Q)^{-1}\Id$.
For an orthonormal basis $\{e_i\}$, it follows that
\[
 \tr R=\sum_i\inner{Q^{-1}We_i}{We_i}
 \geq\frac1{\lambda_{\max}(Q)}\sum_i\norm{We_i}^2
 =\frac{\norm W^2}{\lambda_{\max}(Q)}.
\]
This uses no commutation between $W$ and $Q$.

Suppose first that $h_B\geq h_A/2$.  Since
$W^\circ=h_BU_B-h_AU_A$ and $\tr W=h_B-h_A$,
\[
 h_B(U_B-U_A)=W^\circ-(\tr W)U_A.
\]
Using $\norm{S+T}^2\leq2\norm S^2+2\norm T^2$ and
$\norm{U_A}<1$, we obtain
\begin{align*}
 h_B^2\norm{U_B-U_A}^2
 &\leq2\norm{W^\circ}^2+2(\tr W)^2\norm{U_A}^2\\
 &\leq2\bigl(\norm{W^\circ}^2+(\tr W)^2\bigr)
 \leq2n\norm W^2.
\end{align*}
The last inequality follows from the orthogonal trace decomposition
\[
 \norm W^2=\norm{W^\circ}^2+\frac{(\tr W)^2}{n},
 \qquad
 \norm{W^\circ}^2+(\tr W)^2
 =n\norm W^2-(n-1)\norm{W^\circ}^2\leq n\norm W^2.
\]
Dividing by $h_B$ and using \eqref{eq:lambdaQ}--\eqref{eq:R-lower},
together with $h_A/h_B\leq2$, gives
\begin{align*}
 h_B\norm{U_B-U_A}^2
 &\leq\frac{2n}{h_B}\norm W^2
 \leq\frac{2n\lambda_{\max}(Q)}{h_B}\tr R\\
 &\leq2n\left(1+n(a+1)\frac{h_A}{h_B}\right)\tr R
 \leq2n(1+2n(a+1))\tr R.
\end{align*}

Suppose now that $h_B<h_A/2$.  The triangle inequality gives
$\norm{U_B-U_A}\leq\norm{U_B}+\norm{U_A}<2$, so
\[
 h_B\norm{U_B-U_A}^2<4h_B<2h_A,
\]
whereas the trace decomposition above gives
\[
 \norm W^2\geq\frac{(\tr W)^2}{n}
 =\frac{(h_A-h_B)^2}{n}>\frac{h_A^2}{4n}.
\]
At the same time, \eqref{eq:lambdaQ} implies
\[
 \lambda_{\max}(Q)
 \leq h_B+n(a+1)h_A
 <\frac{h_A}{2}\bigl(1+2n(a+1)\bigr).
\]
Combining the last two bounds with \eqref{eq:R-lower} yields
\[
 \tr R\geq\frac{\norm W^2}{\lambda_{\max}(Q)}
 >\frac{h_A^2/(4n)}{h_A(1+2n(a+1))/2}
 =\frac{h_A}{2n(1+2n(a+1))}.
\]
Multiplying by $4n(1+2n(a+1))$ shows that
$2h_A<4n(1+2n(a+1))\tr R$, which bounds
$h_B\norm{U_B-U_A}^2$ as required.
Thus in both cases
\begin{equation}\label{eq:perspective-bound}
 h_B\norm{U_B-U_A}^2
 \leq4n(1+2n(a+1))\tr R.
\end{equation}
Finally, \eqref{eq:DgR-bound} and \eqref{eq:perspective-bound} give
\[
 h_B\norm{U_B-U_A}^2+\dot q^{ij}(A)R_{ij}
 \leq\left[4n(1+2n(a+1))+\frac{n-1}{n}\right]\tr R
 =C_{n,a}\tr R.
\]
Since $\varepsilon>0$, substitution into \eqref{eq:E-explicit-decomp}
reverses the inequality for the subtracted term and yields
\[
 \mathcal E_z(A,B)\geq(1-\varepsilon C_{n,a})\tr R\geq0.
\]
The final inequality follows from
$\varepsilon<C_{n,a}^{-1}$ in \eqref{eq:epsilon-range} and $R\geq0$.
Recalling $R=WQ^{-1}W$ and $W=B-A$ proves \eqref{eq:global-E-bound}.
\end{proof}

\begin{lemma}[Preservation of the admissible cone]
\label{lem:cone-preservation}
Let $0<\varepsilon<1$, and let $X$ be a smooth closed solution with
speed \eqref{eq:F-explicit} which is initially admissible for the cone
\eqref{eq:Gamma-explicit}.  Then it remains $\Gamma$-admissible
throughout its smooth existence interval.  More precisely, put
\[
 \theta_0=\max_M\frac{\norm{A^\circ(\cdot,0)}^2}{H(\cdot,0)^2}<1,
 \qquad m_0=\min_M F_\varepsilon(\cdot,0)>0.
\]
Then
\begin{equation}\label{eq:preserved-cone-margin}
 \frac{\norm{A^\circ}^2}{H^2}\leq\theta_0,
 \qquad H\geq F_\varepsilon\geq m_0.
\end{equation}
\end{lemma}

\begin{proof}
The ellipticity calculation in Lemma~\ref{lem:speed-properties} uses
only $0<\varepsilon<1$.  Work first on a compact admissible time
interval.  Taking the trace in \eqref{eq:basic-evolution} gives
$\partial_tH=\Delta F_\varepsilon+F_\varepsilon\norm A^2$.
The chain rule and Simons' identity
\[
 \nabla^p\nabla_p h_{ij}=\nabla_i\nabla_jH
 +Hh_i{}^kh_{kj}-\norm A^2h_{ij}
\]
therefore give
\begin{align*}
 (\partial_t-\Delta_{F_\varepsilon})H
 ={}&H|A|_{F_\varepsilon}^2
 +\ddot F_\varepsilon^{ij,kl}\nabla_p h_{ij}\nabla^p h_{kl}\\
 &-\norm A^2\dot F_\varepsilon^{ij}h_{ij}
   +F_\varepsilon\norm A^2.
\end{align*}
Euler homogeneity cancels the last two terms, leaving
\begin{equation}\label{eq:H-evolution}
 (\partial_t-\Delta_{F_\varepsilon})H
 =H|A|_{F_\varepsilon}^2
 +\ddot F_\varepsilon^{ij,kl}\nabla_p h_{ij}\nabla^p h_{kl}.
\end{equation}
Set $U=A^\circ/H$, so that
\[
 \frac{F_\varepsilon}{H}=1-\varepsilon\norm U^2.
\]
Since $(\nabla_pA)^\circ-(\nabla_pH)U=H\nabla_pU$,
equation~\eqref{eq:quotient-second} turns \eqref{eq:H-evolution} into
\[
 (\partial_t-\Delta_{F_\varepsilon})H
 =H|A|_{F_\varepsilon}^2-2\varepsilon H\norm{\nabla U}^2,
\]
where $\norm{\nabla U}^2=\sum_p\norm{\nabla_pU}^2$ in an
orthonormal frame.  Combining this with \eqref{eq:F-evolution}, the
quotient rule yields
\[
 (\partial_t-\Delta_{F_\varepsilon})\frac{F_\varepsilon}{H}
 =\frac2H\dot F_\varepsilon^{ij}(\nabla_iH)
       \nabla_j\frac{F_\varepsilon}{H}
 +2\varepsilon\frac{F_\varepsilon}{H}\norm{\nabla U}^2.
\]
Therefore
\begin{equation}\label{eq:normalized-pinching-evolution}
 \left(\partial_t-\Delta_{F_\varepsilon}
 -\frac2H\dot F_\varepsilon^{ij}(\nabla_iH)\nabla_j\right)\norm U^2
 =-2\frac{F_\varepsilon}{H}\norm{\nabla U}^2\leq0.
\end{equation}
The maximum principle gives $\norm U^2\leq\theta_0$.
Equation~\eqref{eq:F-evolution}
also gives $F_\varepsilon\geq m_0$, since $|A|_{F_\varepsilon}^2\geq0$.
As $F_\varepsilon\leq H$, this proves \eqref{eq:preserved-cone-margin}
on every compact admissible time interval.

These bounds are independent of the endpoint of that interval.  At a
putative smooth first exit from $\Gamma$, continuity would thus give
$H\geq m_0>0$ and $\norm{A^\circ}^2/H^2\leq\theta_0<1$,
contradicting exit.  This proves preservation of admissibility; it does
not assert continuation through a curvature singularity.
\end{proof}

\begin{proof}[Proof of Theorem~\ref{thm:explicit}]
Lemma~\ref{lem:cone-preservation} preserves $\Gamma$-admissibility
throughout the smooth embedded existence interval.  Lemmas
\ref{lem:speed-properties} and \ref{lem:global-shift} verify the
remaining hypotheses of Theorem~\ref{thm:general} with $b=a+1$.
Applying that theorem completes the proof.
\end{proof}

The exterior estimate is now established, and the same speed is concave
on its preserved cone.  We can therefore combine it with the known
interior estimate to prove the two-sided statement.

\begin{proof}[Proof of Corollary~\ref{cor:two-sided}]
The diagonal compactification in Section~\ref{sec:preliminaries}
makes $k(\cdot,\cdot,0)$ bounded.  Since the initial speed has a positive
minimum, $C_{\mathrm{in}}$ is finite.  The diagonal values also give
$\overline{k}\geq\kappa_n\geq H/n>0$, so $C_{\mathrm{in}}>0$.

By Lemmas~\ref{lem:speed-properties} and \ref{lem:cone-preservation},
$F_\varepsilon$ is concave and positive on the convex cone $\Gamma$,
and the flow remains admissible.  The interior estimate of Andrews,
Langford and McCoy~\cite[Corollary~3]{ALM13} therefore gives
$\overline{k}\leq C_{\mathrm{in}}F_\varepsilon$.
Together with Theorem~\ref{thm:explicit}, this proves both curvature
bounds.  Their tangent-ball interpretation
\cite[Proposition~4]{ALM13} gives an interior ball of radius
$(C_{\mathrm{in}}F_\varepsilon)^{-1}$ and an exterior ball of radius
$(aF_\varepsilon)^{-1}$ at each point.  Both radii are at least
$\alpha/F_\varepsilon$, as claimed.
\end{proof}

We finish by recording the strictness of the structural comparison
and by giving closed nonconvex initial hypersurfaces to which the
estimates apply.

\begin{remark}[Strictness at finite shift depth]\label{rem:strictness}
The implication in Proposition~\ref{prop:hierarchy}\textup{(i)} is strict
at every finite shift depth.  Indeed, fix $n\geq2$ and $b\geq0$, choose
$a>\max\{\beta_n,b-1\}$, and then choose $\varepsilon$ satisfying
\eqref{eq:epsilon-range}; this range is nonempty because $a>\beta_n$.
On the convex cone \eqref{eq:Gamma-explicit}, which contains the positive
definite cone, the speed \eqref{eq:F-explicit} is smooth, positive,
homogeneous of degree one, strictly elliptic and non-convex by
Lemma~\ref{lem:speed-properties}.  Lemma~\ref{lem:global-shift} shows that
it satisfies $\mathcal S_{a+1}$.  Since $a+1>b$, every shift allowed by
$\mathcal S_b$ is also allowed by $\mathcal S_{a+1}$:
\[
 z\geq-bF_\varepsilon(A)\geq-(a+1)F_\varepsilon(A).
\]
Thus this non-convex speed satisfies $\mathcal S_b$, as required.
\end{remark}

\begin{remark}[Nonconvex initial hypersurfaces]\label{rem:nonconvex-data}
The cone \eqref{eq:Gamma-explicit} includes negative
principal curvatures.  For example, it contains
$\operatorname{diag}(-\delta,1,\ldots,1)$ for all sufficiently small
$\delta>0$.  It also occurs on closed embedded non-convex hypersurfaces.
For $R>r>0$, consider the thin tube
\[
 X(\theta,\omega)=
 \bigl((R+r\omega_1)\cos\theta,
 (R+r\omega_1)\sin\theta,
 r\omega_2,\ldots,r\omega_n\bigr),
\]
where $\theta\in S^1$ and $\omega\in S^{n-1}$.  With the outward tube
normal, its principal curvatures are
\[
 \frac1r\quad\text{with multiplicity }n-1,
 \qquad
 \frac{\omega_1}{R+r\omega_1}.
\]
The last curvature is negative on the inner side.  As $R/r\to\infty$,
uniformly in $(\theta,\omega)$,
\[
 \frac{\norm{A^{\circ}}}{H}\longrightarrow
 \frac1{\sqrt{n(n-1)}}<1.
\]
Hence every sufficiently thin such tube is $\Gamma$-admissible and
non-convex.  To verify the initial exterior bound, fix one such tube and set
\[
 a_0=\max\left\{0,-\inf_M\frac{\ka(\cdot,0)}{H(\cdot,0)}\right\}<\infty,
\]
where finiteness follows from the diagonal compactification
and $H>0$ on the compact admissible tube.  Choose
\(a>\max\{\beta_n,2a_0\}\), and then choose $\varepsilon$ according
to \eqref{eq:epsilon-range}.  The bound
\(\varepsilon<C_{n,a}^{-1}<1/2\) gives
\(F_\varepsilon>H/2\), and hence
\[
 \ka(\cdot,0)+aF_\varepsilon(\cdot,0)
 \geq -a_0H+\frac a2H>0.
\]
Thus this tube satisfies the initial hypotheses of
Theorem~\ref{thm:explicit}.  To apply
\cite[Theorem~18.18]{ACGL}, choose a smooth symmetric cutoff equal to
one on a neighbourhood of the compact initial principal-curvature image
and supported compactly in the eigenvalue cone corresponding to $\Gamma$.
Multiplying the speed by this cutoff and extending by zero gives a smooth
symmetric function on $\R^n$ which remains strictly monotone on that
neighbourhood.  The cited theorem gives a smooth short-time solution.
After shortening the time interval, its principal curvatures stay where
the cutoff equals one, so it solves the original flow; embeddedness
also persists by smooth dependence on time.  Lemma~\ref{lem:cone-preservation}
then preserves admissibility throughout the smooth embedded existence
interval, and Corollary~\ref{cor:two-sided} applies.

\end{remark}

\begin{remark}[Relation to mean curvature flow, and sharpness of the
parameter range]\label{rem:mcf-sharpness}
On $\Gamma$ the speed satisfies
$F_\varepsilon=\tr+O(\varepsilon)$, so the flow \eqref{eq:flow} with
speed $F_\varepsilon$ is a small perturbation of mean curvature flow.
Theorem~\ref{thm:explicit}, however, does not follow from the
noncollapsing theory of mean curvature flow.  The structural hypotheses
of the two-point maximum principle are not open conditions:
$F_\varepsilon$ is non-convex for every $\varepsilon>0$
(Lemma~\ref{lem:speed-properties}), and ordinary inverse-concavity
controls no negative shift at all
(Proposition~\ref{prop:hierarchy}(iii)).  Nor is there a comparison
principle transporting the two-point estimate from one speed to
another: the exterior bound must be verified as a pointwise matrix
comparison on the whole admissible cone, which is exactly the condition
$\mathcal S_b$, and Lemma~\ref{lem:global-shift} is its first
verification for a speed that is not convex.

The first restriction in \eqref{eq:epsilon-range} is sharp for the
curvature bound \eqref{eq:explicit-cone-gap}.  Indeed, for
$A_t=\frac1n\Id+tD$ with
$D=\operatorname{diag}(-(n-1),1,\ldots,1)/\sqrt{n(n-1)}$, one has
$A_t\in\Gamma$ for $0\leq t<1$ and
\[
 \lambda_{\min}(A_t)+aF_\varepsilon(A_t)
 =\frac1n-t\sqrt{\frac{n-1}{n}}+a(1-\varepsilon t^2)
 \longrightarrow-\beta_n+a(1-\varepsilon)
 \qquad\text{as }t\uparrow1,
\]
so \eqref{eq:explicit-cone-gap} fails on $\Gamma$ once
$a(1-\varepsilon)\leq\beta_n$.  Whether the second restriction
$\varepsilon<C_{n,a}^{-1}$, coming from the global estimate
\eqref{eq:global-E-bound}, can be relaxed is open; more generally it
would be interesting to know which other concave speeds satisfy
$\mathcal S_b$ for some $b>0$ on a cone allowing negative principal
curvatures.
\end{remark}

%%%%%%%%%%%%%%%%%
\section*{Declarations}

%%%%%%%%%%%%%%%%%%%
%%%%%%%%%%%%%%%%%%%
{\bf AI usage.} 
The authors used AI models, in particular GPT-5.6 Sol and Kimi K3, to help construct the example in Theorem~\ref{thm:explicit} and refine the exposition. The authors take full responsibility for the content of the manuscript.

%%%%%%%%%%%%%%%%%%%%%
{\bf Acknowledgements.}
W. Sheng was partially supported by National Key R$\&$D Program of China (No. 2022YFA1005500) and Natural Science Foundation of China under Grant No. 12571063.

%%%%%%%%%%%%%%%%%%%%%%%
%%%%%%%%%%%%%%%%%%%%%%%
\textbf{Conflict of interest.} The authors declare that they have no
conflict of interest.

\textbf{Data availability.} No datasets were generated or analysed
during the current study.

%%%%%%%%%%%%%%%%%%%%%%%%%%%

\end{document}